\documentclass[11pt]{amsart}
\usepackage{amssymb,amsmath,amsthm}
\usepackage{enumerate}
\usepackage{fullpage}

\newtheorem{thm}{Theorem}

\theoremstyle{definition}

\newtheorem{defn}[thm]{Definition}

\begin{document}
\title[Dissipative realizations of contractive functions]{Does every contractive analytic function in a polydisk have a dissipative n-dimensional scattering realization?}
\author{Michael T. Jury}

\address{Department of Mathematics\\
  University of Florida\\
  Box 118105\\
  Gainesville, FL 32611-8105\\
  USA}

\email{mjury@ufl.edu}
\thanks{Research partially supported by NSF grant DMS 1101134.}

\date{\today}

\begin{abstract}
  No.
\end{abstract}
\maketitle
The title question was posed by D. Kalyuzhnyi-Verbovetskyi \cite[Problem 1.3]{unsolved}.  Let $L(\mathcal H, \mathcal K)$ denote the set of all bounded linear operators between a pair of Hilbert spaces $\mathcal H, \mathcal K$, and let $\mathbb D^n$ and $\mathbb T^n$ denote the open unit polydisk, and the unit $n$-torus, respectively.

\begin{defn}\label{def:scatter}
  An {\em dissipative nD scattering system} is a tuple
  \begin{equation}
    \alpha =(n; {\bf A}, {\bf B}, {\bf C},{\bf D}; \mathcal X, \mathcal U,\mathcal Y)
  \end{equation}
where:
\begin{itemize}
\item[i)] $n\geq 1$ is an integer;
\item[ii)] $\mathcal X, \mathcal U, \mathcal Y$ are Hilbert spaces;
\item[iii)]  ${\bf A}, {\bf B}, {\bf C},{\bf D}$ are $n$-tuples of operators (so ${\bf A}=(A_1, \dots A_n)$, etc.) with
  \begin{equation}
    A_k\in L(\mathcal X,\mathcal X), \quad B_k\in L(\mathcal U, \mathcal X), \quad C_k\in L(\mathcal X,\mathcal Y),\quad D_k \in L(\mathcal U, \mathcal Y);
  \end{equation}
\item[iv)] The operator $\zeta{\bf G}\in L(\mathcal X \oplus\mathcal U, \mathcal X\oplus \mathcal Y)$ is contractive for all $\zeta$ in the unit $n$-torus $\mathbb T^n$, where 
  \begin{equation}
    \zeta{\bf G}:= \sum_{k=1}^n \zeta_k G_k
  \end{equation}
and the $G_k$ are the $2\times 2$ block operators
\begin{equation}
  G_k= \begin{pmatrix} A_k & B_k \\
                       C_k & D_k  \end{pmatrix}
\end{equation}
\end{itemize}
\end{defn}
Given such a system, its {\em transfer function} is the $L(\mathcal U, \mathcal Y)$-valued function 
\begin{equation}
  \theta_\alpha(z) = z{\bf D} +z{\bf C}(I_{\mathcal X} -z{\bf A})^{-1} z{\bf B}.
\end{equation}
defined for all $z\in\mathbb D^n$. It is shown in \cite{KV-1} that the transfer function $\theta_\alpha$ is a {\em contractive operator function}; that is, it is analytic in the unit polydisk $\mathbb D^n$ and satisfies
\begin{equation}
  \|\theta_\alpha(z)\|_{L(\mathcal U, \mathcal Y)} \leq 1
\end{equation}
for all $z\in\mathbb D^n$.  The question is then whether every contractive operator function in $\mathbb D^n$, vanishing at the origin, is such a transfer function. The answer is known to be ``yes'' when $n=1$ or $2$, and in fact a stronger result is true: ${\bf G}$ can be chosen so that $\zeta {\bf G}$ is unitary (that is, the scattering system is {\em conservative}). It was also known that when $n=3$, there exist contractive operator functions which do not have conservative realizations; this is due to the failure of von Neumann's inequality in three variables.  (See \cite{unsolved, KV-1} for a discussion.)  In this note we show the answer is still ``no'' in the dissipative case when $n=3$, and give an explicit counterexample (in the scalar case $\mathcal U=\mathcal Y=\mathbb C$). 

We first show that any polynomial with a dissipative realization must satisfy a restricted form of von Neumann's inequality. Let $\mathcal T$ denote the set of all $n$-tuples of commuting operators ${\bf T}=(T_1, \dots T_n)$ on Hilbert space satisfying the following condition: whenever ${\bf X}=(X_1, \dots X_n)$ is an $n$-tuple of operators satisfying 
\begin{equation}
  \|\sum_{k=1}^n z_k X_k\|\leq 1
\end{equation}
for all $z=(z_1, \dots z_n)\in \mathbb D^n$, then 
\begin{equation}
  \|\sum_{k=1}^n T_k\otimes X_k\|_{L(\mathcal H \otimes \mathcal K)}\leq 1
\end{equation}
where the $T_k$ act on $\mathcal H$ and the $X_k$ act on $\mathcal K$. 

It is easy to see that the ${\bf T}$ satisfying this condition must be commuting contractions, but when $n\geq 3$ it is known that not every $n$-tuple of commuting contractions belongs to $\mathcal T$.
\begin{thm}\label{thm:restricted-vn}
  If $p$ is a polynomial which can be realized as the transfer function of a dissipative nD scattering system, then
  \begin{equation}
    \|p({\bf T})\|\leq 1
  \end{equation}
for all ${\bf T}\in \mathcal T$.
\end{thm}
We will say such $p$ satisfy the {\em restricted von Neumann inequality}.
\begin{proof}[Proof of Theorem~\ref{thm:restricted-vn}]
  Suppose $p$ is a polynomial vanishing at $0$ and $p=\theta_\alpha$ for some $\alpha$ as in Definition~\ref{def:scatter}; we work only in the scalar case $\mathcal U=\mathcal Y=\mathbb C$.  First note that since analytic functions in the polydisk satisfy a maximum principle relative to $\mathbb T^n$, the dissipativity condition (iv) implies
  \begin{equation}
    \|z{\bf G}\|:= \|\sum_{k=1}^n z_k G_k\|\leq 1
  \end{equation}
for all $z\in\mathbb D^n$. Then by definition, if ${\bf T}\in \mathcal T$, we have
\begin{equation}
  \|\sum_{k=1}^n T_k\otimes G_k\|\leq 1.
\end{equation}
Next, we recall the classical fact that if 
\begin{equation}
  F =
  \begin{pmatrix}
    W & X \\
    Y & Z \end{pmatrix}
\end{equation}
is a block operator and $\|F\|<1$, then the linear fractional operator
\begin{equation}
  Z+Y(I-W)^{-1}X
\end{equation}
is contractive. Now apply this to the block operator
\begin{equation}
  r{\bf T}\cdot {\bf G} = 
  \begin{pmatrix}
     r{\bf T}\cdot {\bf A} &  r{\bf T}\cdot {\bf B} \\
      r{\bf T}\cdot {\bf C} &  r{\bf T}\cdot {\bf D}  \end{pmatrix}
\end{equation}
where ${\bf T}\cdot {\bf A} :=\sum_{k=1}^n T_k\otimes A_k$, etc., and $r<1$. We conclude that if ${\bf T}\in\mathcal T$, then the linear fractional operator
\begin{equation}\label{eqn:lft-T}
  r{\bf T}\cdot {\bf D} + r{\bf T}\cdot {\bf C}(I_{\mathcal H\otimes \mathcal X} -r{\bf T}\cdot {\bf A})^{-1} r{\bf T}\cdot {\bf B}
\end{equation}
is contractive for all $r<1$. But it is straightforward to check that, since $p$ is assumed to be given by the transfer function realization
\begin{equation}\label{eqn:p-transfer}
p(z) = z{\bf D} +z{\bf C}(I_{\mathcal X} -z{\bf A})^{-1} z{\bf B}, 
\end{equation}
the expression (\ref{eqn:lft-T}) is equal to $p(r{\bf T})$ (This can be done by expanding (\ref{eqn:p-transfer}) in a power series, substituting $r{\bf T}$ for $z$, and comparing coefficients with the expansion of (\ref{eqn:lft-T}) in powers of $rT_1, \dots rT_n$.) But then $\|p(r{\bf T})\|\leq 1$ for all ${\bf T}\in\mathcal T$ and $r<1$, which suffices to establish the theorem. 
\end{proof}
It follows that any contractive polynomial which fails the restricted von Neumann inequality will fail to have a dissipative realization.  In fact, the counterexample to the classical von Neumann inequality produced by Kaijser and Varopoulos is, it turns out, also a counterexample to the restricted inequality, as we now show. The computations are taken from a closely related example considered in \cite{me}.

Let $e_1, \dots e_5$ denote the standard basis of $\mathbb C^5$.  Consider the unit vectors
\begin{align*}
v_1 &=\frac{1}{\sqrt{3}}(-e_2+e_3+e_4)\\
v_2 &=\frac{1}{\sqrt{3}}(e_2-e_3+e_4)\\
v_3 &=\frac{1}{\sqrt{3}}(e_2+e_3-e_4)\\
\end{align*}
The Kaijser-Varopoulos contractions are the commuting $5\times 5$ matrices $T_1, T_2, T_3$ defined by
$$
T_j = e_{j+1} \otimes e_1 +e_5\otimes v_j
$$
If $p$ is the polynomial
\begin{equation}\label{eqn:kv-p}
  p(z_1, z_2, z_3) = \frac15(z_1^2+z_2^2+z_3^2-2z_1z_2-2z_1z_3-2z_2z_3)
\end{equation}
then it is known that $\sup_{\zeta\in\mathbb T^3}|p(\zeta)|=1$ but 
\begin{equation}
  \|p({\bf T})\|=\frac{3\sqrt{3}}{5}>1,
\end{equation}
so $p$ fails the classical von Neumann inequality \cite{Var-74}. To show that this $p$ fails the restricted von Neumann inequality, we show that already this ${\bf T}$ belongs to $\mathcal T$; that is, if $X_1, X_2, X_3$ are operators which satisfy
\begin{equation}\label{E:minell1def}
\|z_1 X_1 +z_2 X_2 +z_3 X_3\|\leq 1
\end{equation}
for all $z\in\mathbb D^n$, then $\|\sum_{k=1}^3 T_k\otimes X_k\|\leq 1$.  To see this, we compute and find
\begin{equation}\label{E:atensort}
T_1\otimes X_1 + T_2\otimes X_2+ T_3\otimes X_3 =\begin{pmatrix}
0 & 0 & 0 & 0 & 0 \\
X_1 & 0 & 0 & 0 & 0 \\
X_2 & 0 & 0 & 0 & 0 \\
X_3 & 0 & 0 & 0 & 0 \\
0 & Y_1 & Y_2 & Y_3 & 0 \end{pmatrix}
\end{equation}
where 
\begin{align*}
Y_1 &=\frac{1}{\sqrt{3}}(-X_1+X_2+X_3)\\
Y_2 &=\frac{1}{\sqrt{3}}(X_1-X_2+X_3)\\
Y_3 &=\frac{1}{\sqrt{3}}(X_1+X_2-X_3)\\
\end{align*}
The norm of the matrix (\ref{E:atensort}) is equal to the maximum of the norms of the first column and the last row.  By (\ref{E:minell1def}), we have $\|\pm X_1\pm X_2\pm X_3\|\leq 1$ for all choices of signs, so the last row of (\ref{E:atensort}) has norm at most $1$.  To say that the first column has norm at most 1 amounts to saying that 
\begin{equation}\label{E:column-contraction}
I-\sum_{k=1}^n X_k^* X_k \geq 0.
\end{equation}
This may be seen by averaging:  by (\ref{E:minell1def}), the matrix valued function
$$
I -\sum_{i,j=1}^n \zeta_i\overline{\zeta_j} X_j^* X_i 
$$
is positive semidefinite on $\mathbb T^n$.  Integrating against normalized Lebesgue measure on $\mathbb T^n$ gives (\ref{E:column-contraction}).

There is a general principle that transfer function realizations should be equivalent to von Neumann-type inequalities. Some recent, general results in this direction may be found in \cite{me, MP}.
\bibliographystyle{plain} 
\bibliography{realization} 

\end{document}